\documentclass[12pt,reqno]{article}
\usepackage{amsfonts}
\usepackage{amsmath, amsthm, amsfonts, amssymb, color}
\usepackage{mathrsfs}
\usepackage{fixmath}

\allowdisplaybreaks

\allowdisplaybreaks

\newtheorem{thm}{Theorem}[section]

\newtheorem{lem}[thm]{Lemma}

\newtheorem{ass}[thm]{Assumption}
\newtheorem{rem}{Remark}

\theoremstyle{definition}

\numberwithin{equation}{section}
\begin{document}




\title{Strong convergence rate of the truncated Euler-Maruyama method for stochastic differential delay equations with Poisson jumps}

\author{
{\bf  Shuaibin Gao$^a$, \, Junhao Hu$^a$, \,  Li Tan$^{b,c}$\, and \, Chenggui Yuan$^d$}\\
 \footnotesize{$^{a}$ College of Mathematics and statistics, South-Central University For Nationalities,
Wuhan 430074, China}\\
 \footnotesize{$^{b}$ School of Statistics, Jiangxi University of Finance and Economics, Nanchang, Jiangxi, 330013,  China}\\
 \footnotesize{$^{c}$ Research Center of Applied Statistics, Jiangxi University of Finance and Economics,}\\
  \footnotesize{ Nanchang, Jiangxi, 330013,  China}\\
\footnotesize{$^d$ Department of Mathematics, Swansea University, Swansea, SA2 8PP, U. K. }\\
\footnotesize{shuaibingao@163.com, junhaohu74@163.com, tltanli@126.com, C.Yuan@swansea.ac.uk}}


\maketitle

\begin{abstract}
In this paper, we study
a class of super-linear stochastic differential delay equations with Poisson jumps (SDDEwPJs). The  convergence and  rate of the convergence of the truncated Euler-Maruyama  numerical solutions  to SDDEwPJs are investigated under the generalized Khasminskii-type condition.
\end{abstract}




\section{Introduction}

Since the establishment of stochastic differential equations (SDEs) driven by Brownian motions,  many scholars have contributed to study properties of SDEs,  for example, \cite{1,3,16,19} and references therein.
We can observe that the stochastic systems are widely applied in many fields such as biology, chemistry, finance and economy.
When studying the realistic models, it is found that the real state is not only related to the present state, but also related to the past state.
Stochastic differential delay equations (SDDEs) are used to describe such systems \cite{4,5,20}.
Moreover, if  an emergency occurs, its impact on the systems must be taken into account.
For example, the sudden outbreak of the new coronavirus has a huge impact on the global economy, leading to a shock in the stock market.
Hence, SDEs with jumps which take both the continuous and discontinuous random effects into consideration are studied to analyze these situations \cite{11,21,27a}. In this paper, we will take the delay and jumps into the cosideration, i.e.  we shall study SDDEs with jumps  \cite{15,28,30}.

Generally speaking, the true solutions of many equations can not be calculated, so it is meaningful to investigate the numerical solutions.
For instance, the explicit Euler-Maruyama (EM) schemes are very popular to approximate the true solutions \cite{16}.
However, when the coefficients grow super-linearly, Hutzenthaler et al. in \cite{13} proved that the $p$th moments of the EM approximations diverge to infinity for all $p \in [1,\infty)$.
Thus, many implicit methods have been put forward to approximate the solutions of the equations with nonlinear growing coefficients \cite{2,10,24,27}.
In addition, since the explicit schemes require less computation,  some modified EM methods have also been established for   nonlinear stochastic equations \cite{14,18,25,26}.
In particularly, the truncated EM method was originally proposed by Mao in \cite{22} with drift and diffusion coefficients growing super-linearly. The rate of convergence  of the truncated EM method was obtained in \cite{23}.
Afterward, there are  many papers to study the truncated EM method for  stochastic equations whose coefficients grow super-linearly, and we refer to  \cite{7,8,9,12,17} and references therein. Additionally, there are many results on the numerical solutions for SDE with jumps and SDDEs with jumps. For example,
the convergence in probability of the EM method for SDDEs with jumps was discussed in \cite{15}.
The strong convergence of the EM method for SDDEs with jumps as well as the modified split-step backward Euler method approximation to the true solution was presented in \cite{30}. The semi-implicit Euler method for SDDEs with jumps is convergent with strong order $1/2$ \cite{28}.
However, there are few papers studying the numerical solutions of the super-linear SDDEs with Poisson jumps (SDDEwPJs) whose all three coefficients might grow super-linearly.
Therefore, in this paper, we will investigate the strong convergence rate of the truncated EM method for super-linear SDDEwPJs in $\mathcal{L}^p$($p>0$) sense.

This paper is organized as follows. We will introduce some necessary notation  in Section 2. The rate of convergence  in $\mathcal{L}^p$ for $p \geq 2$ will be discussed in Section 3. In Section 4, the rate of convergence   in $\mathcal{L}^p$ for $0<p<2$ will be presented. Section 5 contains an example to illustrate that our main result could cover a large class of super-linear SDDEwPJs.

\section{Mathematical preliminaries}

\quad Throughout this paper, unless otherwise specified, we use the following notation.
If $A$ is a vector or matrix, its transpose is denoted by $A^T$.
For $x\in \mathbb{R}^n$,  $|x|$ denotes its Euclidean norm.
If $A$ is a matrix, we let $|A| = \sqrt{\text{trace}(A^TA)}$ be its trace norm.
By $A \le 0$ and $A<0$, we mean $A$ is non-positive and negative definite, respectively.
If both $a, b$ are real numbers, then $a\wedge b = \text {min}\{a, b\}$ and $a\vee b=\text{max}\{a,b\}$.
Let $\left \lfloor a \right \rfloor$ denote the largest integer which does not exceed $a$.
Let $\mathbb{R} _+ = [0,+\infty)$ and $\tau > 0$.
Denote by $\mathscr{C}([-\tau,0];\mathbb{R}^n)$ the family of continuous functions $\varphi$ from $[-\tau,0] $ to $\mathbb{R}^n$ with the norm $\|\varphi\|=\text{sup}_{-\tau \le \theta \le 0}|\varphi (\theta)|$.
If $H$ is a set, denote by $\textbf{1}_H$ its indicator function; that is,  $\textbf{1}_H(x)=1$ if $x \in H$ and $\textbf{1}_H(x)=0$ if $x \notin H$.
Let $C$ stand for a generic positive real constant whose value may change in different appearances.

Let $\big ( \Omega, \mathcal{F}, \{\mathcal{F}_t\}_{t\ge 0}, \mathbb{P} \big )$ be a complete probability space with a filtration  $\{\mathcal{F}_t\}_{t\ge 0}$ satisfying the usual conditions (i.e. it is increasing and right continuous while $ \mathcal{F}_0$ contains all $\mathbb{P}$-null sets).
Let $\mathbb{E}$ denote the probability expectation with respect to $\mathbb{P}$.
For $p >0$, $\mathcal{L}^p =\mathcal{L}^p( \Omega, \mathcal{F}, \{\mathcal{F}_t\}_{t\ge 0}, \mathbb{P} \big )$ denotes the space of random variables $X$ with a norm $|X|_p :=(\mathbb{E}|X|^p)^{1/p}<\infty$. Let $\mathscr{C}_{\mathcal{F}_0}^b ([-\tau,0];\mathbb{R}^n)$ denote the family of all bounded, $\mathcal{F}_0$-measurable, $\mathscr{C}([-\tau,0];\mathbb{R}^n)$-valued random variables.
Let $B(t)  = \big (B_1(t),\cdots,B_m(t) \big ) ^T$ be an $m$-dimensional Brownian motion defined on the probability space.
Let $N(t)$ be a scalar Poisson process with the compensated Poisson process $\widetilde{N}(t) = N(t)-\lambda t$, where the parameter $\lambda > 0$ is the jump intensity. Furthermore, we assume that $B(t)$ and $N(t)$ are independent.


In this paper, we study the truncated EM method for super-linear SDDEwPJs of the form
\begin{equation}\label{dx1}
d x(t)=f(x(t),x(t-\tau))dt +g(x(t),x(t-\tau))d B(t)+h(x(t^-),x((t-\tau)^-))d N(t),~~~
t\geq 0,
\end{equation}
with the initial value
\begin{equation}\label{dx2}
\xi =\{x(\theta):-\tau \leq \theta \leq 0\}\in \mathscr{C}_{\mathcal{F}_0}^b ([-\tau,0];\mathbb{R}^n),
\end{equation}
where $x(t^-)=\lim_{s\downarrow t} x(s)$.
Here, $f: \mathbb{R}^n\times \mathbb{R}^n \rightarrow \mathbb{R}^n$, $g: \mathbb{R}^n\times \mathbb{R}^n  \rightarrow \mathbb{R}^{n \times m}$, $h: \mathbb{R}^n\times \mathbb{R}^n  \rightarrow \mathbb{R}^n$. In order to estimate the convergence rate of the truncated EM method, we assume that the initial value $\xi$ is $\gamma$-H\"{o}lder continuous, which is a standard constraint \cite{9,30}.
\begin{ass}\label{2.1}
There exist constants $\bar{K} >0$ and $\gamma \in (0,1]$ such that
\begin{equation}
\begin{split}
|\xi(\bar{t})- \xi(\bar{s})| \leq \bar{K}|\bar{t}-\bar{s}|^\gamma,~~~-\tau \leq \bar{s} < \bar{t}\leq 0.
\end{split}
\end{equation}
\end{ass}

\section{Rate of Convergence  in $\mathcal{L}^p$($p \geq 2$)}

Now, in order to obtain  the rate of convergence for  the truncated EM method for \eqref{dx1} in $\mathcal{L}^p$($p \geq 2$) sense, we need to impose the following assumptions on coefficients.\\
\begin{ass}\label{a1}
There exist constants $K_1 >0$ and $\beta \geq 0$ such that
\begin{equation}
\begin{split}
&|f(x,y)-f(\bar{x},\bar{y})|\vee |g(x,y)-g(\bar{x},\bar{y})| \\
&\leq K_1(1+|x|^\beta +|y|^\beta +|\bar{x}|^\beta +|\bar{y}|^\beta)(|x-\bar{x}|+|y-\bar{y}|)
\end{split}
\end{equation}
and
\begin{equation}
|h(x,y)-h(\bar{x},\bar{y})|\leq K_1(|x-\bar{x}|+|y-\bar{y}|)
\end{equation}
for any $x,y,\bar{x},\bar{y} \in \mathbb{R} ^n $.
\end{ass}
By Assumption \ref{a1}, we obtain
\begin{equation}
\begin{split}
|f(x,y)|\vee |g(x,y)| \leq (4K_1+|f(0,0)|+|g(0,0)|)(1+|x|^{\beta +1} +|y|^{\beta +1})
\end{split}
\end{equation}
and
\begin{equation}\label{j1}
|h(x,y)|\leq (K_1 +|h(0,0)|)(1+|x|+|y|)
\end{equation}
for any $x,y \in \mathbb{R} ^n $.

Before stating the next assumption, we need more notation. Let $\mathcal{U}$ denote the family of continuous functions $U:\mathbb{R}^n \times \mathbb{R}^n\rightarrow \mathbb{R}_+$ such that for any $b>0$, there exists a constant $\kappa_b >0$ satisfying
\begin{equation}\label{zs1}
U(x,\bar{x})\leq \kappa_b|x-\bar{x}|^2
\end{equation}
for any $x,\bar{x}\in \mathbb{R} ^n $ with $|x|\vee|\bar{x}|\leq b$. \\
\begin{ass}\label{a2}
There exist constants $K_2 >0,$ $\bar{\eta} > 2$ and $U\in \mathcal{U} $ such that
\begin{equation}
\begin{split}
&(x-\bar{x})^T (f(x,y)-f(\bar{x},\bar{y}))+\frac{\bar{\eta} -1}{2} |g(x,y)-g(\bar{x},\bar{y})|^2\\
&\leq K_2(|x-\bar{x}|^2+|y-\bar{y}|^2)-U(x,\bar{x})+U(y,\bar{y})
\end{split}
\end{equation}
for any $x,y,\bar{x},\bar{y} \in \mathbb{R} ^n$.
\end{ass}
\begin{rem}
We use an example to illustrate the necessity of setting $U(\cdot,\cdot)$. Let $f(x,y)=-5x^3+\frac{1}{8}|y|^{\frac{5}{4}}+2x$, $g(x,y)=\frac{1}{2}|x|^{\frac{3}{2}}+y$ for $x,y\in \mathbb{R}$. We could observe that there is no $K_2$ satisfying
\begin{equation*}
\begin{split}
(x-\bar{x})^T (f(x,y)-f(\bar{x},\bar{y}))+\frac{\bar{\eta} -1}{2} |g(x,y)-g(\bar{x},\bar{y})|^2\leq K_2(|x-\bar{x}|^2+|y-\bar{y}|^2),
\end{split}
\end{equation*}
but Assumption \ref{a2} is satisfied. The detailed proof will be provided in Section 5.
\end{rem}
\begin{ass}\label{a3}
There exist constants $K_3 >0$ and $\bar{p} > \bar{\eta} > 2$ such that
\begin{equation}
\begin{split}
x^T f(x,y) +\frac{\bar{p} -1}{2} |g(x,y)|^2
\leq K_3(1+|x|^2 +|y|^2)
\end{split}
\end{equation}
for any $x,y \in \mathbb{R} ^n $.
\end{ass}

By using the standard method, we could derive that the  moment of the true solution is bounded, i.e.
\begin{lem}\label{L4}
Let Assumption \ref{a1} and \ref{a3} hold. Then SDDEwPJs \eqref{dx1} has a unique global solution $x(t)$. In addition, for any $q\in[2,\bar{p})$,
\begin{equation}
\sup_{0\leq t\leq T} \mathbb{E}|x(t)|^q <\infty,~~~\forall T>0.
\end{equation}
\end{lem}

To our best knowledge, there are few results about the strong convergence of the super-linear SDDEwPJs.  On the other hand,  the truncated EM method developed in \cite{22} is a kind of useful tool to deal with the super-linear terms. Therefore, in this paper, we will show that the truncated EM solutions will converge to the true solution in $\mathcal{L}^p$($p >0$) sense.

To define the truncated EM scheme, we first choose a strictly increasing continuous function $\varphi (r) :\mathbb{R} _+ \rightarrow \mathbb{R}_+$ such that $\varphi (r) \to \infty$ as $r \rightarrow \infty$ and
\begin{equation}
\sup_{|x|\vee|y| \le r} \left ( |f(x,y)| \vee |g(x,y)| \right ) \le \varphi (r),~~~\forall r\geq 1.
\end{equation}
Let $\varphi ^ {-1}$ denote the inverse function of $\varphi$. Thus, $\varphi ^ {-1}$ is a strictly increasing continuous function from $[\varphi (1),\infty)$ to $\mathbb{R}_+$. Then, we choose $K_0 \geq (1\vee \varphi(1))$ and a strictly decreasing function $\alpha: (0,1] \rightarrow (0,\infty)$ such that
\begin{equation}\label{y3}
\lim_{\Delta \rightarrow 0} \alpha(\Delta) =\infty,~~\Delta ^ \frac {1}{4} \alpha(\Delta) \le K_0,~~~\forall \Delta\in (0,1].
\end{equation}
For example, we could choose
\begin{equation}
\alpha(\Delta) =K_0 \Delta ^{-\varepsilon},~~~\forall \Delta\in (0,1],
\end{equation}
for some $\varepsilon \in (0,1/4]$.
For a given step size $\Delta \in (0,1]$, define the truncated mapping $\pi _{\Delta}: \mathbb{R} ^n \rightarrow \mathbb{R} ^n$ by
\begin{equation}\label{zs6}
\pi _{\Delta}(x)=\left ( |x| \wedge \varphi ^ {-1} (\alpha(\Delta)) \right ) \frac {x}{|x|},
\end{equation}
where we let $\frac {x}{|x|} =0$ when $x=0$. It is easy to see that the truncated mapping $\pi _{\Delta}$ maps $x$ to itself when $|x| \leq\varphi ^ {-1} (\alpha(\Delta))$ and to
$\varphi ^ {-1} (\alpha(\Delta))  \frac {x}{|x|}$ when $|x| \geq\varphi ^ {-1} (\alpha(\Delta))$. We now define the truncated functions
\begin{equation}
f_\Delta (x,y)=f(\pi _{\Delta}(x),\pi _{\Delta}(y)),~~g_\Delta (x,y)=g(\pi _{\Delta}(x),\pi _{\Delta}(y)),
\end{equation}
for any $x,y \in \mathbb{R} ^n$.
By definition, we could easily find that
\begin{equation}\label{y1}
|f_\Delta (x,y)|\vee|g_\Delta (x,y)|\leq\varphi(\varphi ^ {-1} (\alpha(\Delta)))\leq \alpha(\Delta),~~~ \forall x,y \in \mathbb{R} ^n.
\end{equation}
Moreover, we could obtain for any $x,y \in \mathbb{R} ^n$ that
\begin{equation}\label{zs2}
|\pi _{\Delta}(x)| \leq |x|, ~~~~~~|\pi _{\Delta}(x)-\pi _{\Delta}(y)|\leq|x-y|,~~~ \forall x,y \in \mathbb{R} ^n.
\end{equation}
By  Assumption \ref{a1}, we derive that for any $x,y \in \mathbb{R} ^n$
\begin{equation}\label{zs3}
\begin{split}
&|f_\Delta(x,y)-f_\Delta(\bar{x},\bar{y})|\vee |g_\Delta(x,y)-g_\Delta(\bar{x},\bar{y})| \\
&\leq K_1(1+|x|^\beta +|y|^\beta +|\bar{x}|^\beta +|\bar{y}|^\beta)(|x-\bar{x}|+|y-\bar{y}|).
\end{split}
\end{equation}
If  Assumption \ref{a3} hold. Then, for any $\Delta \in (0,1],  x,y \in \mathbb{R} ^n$, one has that
\begin{equation}\label{y2}
\begin{split}
x^T f_\Delta(x,y) +\frac{\bar{p} -1}{2} |g_\Delta(x,y)|^2
\leq 3K_3\left(\frac{1}{\varphi ^{-1}(\alpha(1))}\vee 1\right )(1+|x|^2 +|y|^2).
\end{split}
\end{equation}

Let us now introduce the discrete-time truncated EM numerical scheme to approximate the true solution of \eqref{dx1}. For some positive integer $M$, we take step size $\Delta =\tau /M$. Obviously, when we choose $M$ sufficiently large, $\Delta$ will become sufficiently small. Define $t_k = k\Delta$ for $k=-M, -M+1, -M+2, \cdots, -1, 0, 1, 2, \cdots$. Set $X_\Delta(t_k)=\xi(t_k)$ for $k=-M, -M+1, -M+2, \cdots, -1, 0$ and then form
\begin{equation}\label{ls1}
\begin{split}
X_\Delta (t_{k+1})=&X_\Delta (t_{k})+f_\Delta (X_\Delta (t_{k}),X_\Delta (t_{k-M}))\Delta +g_\Delta (X_\Delta (t_{k}),X_\Delta (t_{k-M}))\Delta B_k \\
&+h(X_\Delta (t_{k}^-),X_\Delta (t_{k-M}^-))\Delta N_k
\end{split}
\end{equation}
for $k=0, 1, 2, \cdots$, where $\Delta B_k = B(t_{k+1})-B(t_k)$, $\Delta N_k = N(t_{k+1})-N(t_k)$. As usual, there are two kinds of the continuous-time truncated EM solutions. The first one is that:
\begin{equation}\label{ls2}
\bar{x} _\Delta (t)=\sum _{k=-M}^{\infty} X_\Delta (t_{k})\textbf{I}_{\left [ t_k, t_{k+1}\right )}(t),
\end{equation}
where every sample path is not continuous. The second one is defined as follows:
\begin{equation}\label{ls3}
\begin{split}
x _\Delta (t)=&\xi(0) + \int _0^t f_\Delta (\bar{x} _\Delta (s), \bar{x} _\Delta (s-\tau))ds + \int _0^t g_\Delta (\bar{x} _\Delta (s), \bar{x} _\Delta (s-\tau))dB(s)\\
&+\int _0^t h(\bar{x} _\Delta (s^-), \bar{x} _\Delta ((s-\tau)^-))dN(s),
\end{split}
\end{equation}
which the  sample path is continuous. It is easy to see $X_\Delta (t_{k})=\bar{x}_\Delta (t_k)=x_\Delta (t_k)$. Aditionally, $x_\Delta (t)$ is an It\^{o} process on $t\geq 0$ with its It\^{o} differential:
\begin{equation}
\begin{split}
dx _\Delta (t)=&f_\Delta (\bar{x} _\Delta (t),\bar{x} _\Delta (t-\tau))dt + g_\Delta (\bar{x} _\Delta (t),\bar{x} _\Delta (t-\tau))dB(t)\\
&+h(\bar{x} _\Delta (t^-), \bar{x} _\Delta ((t-\tau)^-))dN(t).
\end{split}
\end{equation}

We now prepare some useful lemmas.
Before stating the next lemma,  we define
\begin{equation*}
\kappa(t)=\lfloor t/\Delta\rfloor \Delta,~~~\forall -\tau \leq t \leq T,
\end{equation*}
where $\lfloor \cdot\rfloor$ is the integer part of a number.
\begin{lem}\label{L6}
For any $\Delta \in (0,1] $ and $t\in [0,T]$, we have
\begin{equation}\label{3.26}
\begin{split}
&\mathbb{E} \left ( |x_\Delta (t)-\bar{x}_\Delta (t)|^{\tilde{p}} \big|\mathcal{F}_{\kappa(t)}\right )\\
&\leq c_{\tilde{p}} \left((\alpha(\Delta))^{\tilde{p}}\Delta ^{\frac{\tilde{p}}{2}} +\Delta\right)(1+\mathbb{E}|\bar{x}_\Delta (t)|^{\tilde{p}}+\mathbb{E}|\bar{x}_\Delta (t-\tau)|^{\tilde{p}}),~~~\tilde{p}\geq 2,
\end{split}
\end{equation}
and
\begin{equation}\label{3.27}
\begin{split}
&\mathbb{E} \left ( |x_\Delta (t)-\bar{x}_\Delta (t)|^{\tilde{p}} \big|\mathcal{F}_{\kappa(t)}\right )\\
&\leq c_{\tilde{p}} (\alpha(\Delta))^{\tilde{p}}\Delta ^{\frac{\tilde{p}}{2}} (1+\mathbb{E}|\bar{x}_\Delta (t)|^{\tilde{p}}+\mathbb{E}|\bar{x}_\Delta (t-\tau)|^{\tilde{p}}),~~~0<\tilde{p}< 2,
\end{split}
\end{equation}
where $c_{\tilde{p}}$ is a positive constant which is independent of $\Delta$.
\end{lem}
\begin{proof}
Fix any $\tilde{p}\geq 2$. Then by the H\"older inequality and the Burkholder-Davis-Gundy inequality, we derive from  \eqref{y1} that
\begin{equation*}
\begin{split}
&\mathbb{E} \left ( |x_\Delta (t)-\bar{x}_\Delta (t)|^{\tilde{p}} \big|\mathcal{F}_{\kappa(t)}\right )\\
&=\mathbb{E} \Big (\big|\int _{\kappa(t)}^t f_\Delta (\bar{x} _\Delta (s), \bar{x} _\Delta (s-\tau))ds + \int _{\kappa(t)}^t g_\Delta (\bar{x} _\Delta (s), \bar{x} _\Delta (s-\tau))dB(s)\\
&~~~+\int _{\kappa(t)}^t h(\bar{x} _\Delta (s^-), \bar{x} _\Delta ((s-\tau)^-))dN(s)\big|^{\tilde{p}} \Big|\mathcal{F}_{\kappa(t)}\Big )\\
&\leq c_{\tilde{p}}\Big ( \mathbb{E} \Big (\big |\int _{\kappa(t)}^t f_\Delta (\bar{x} _\Delta (s), \bar{x} _\Delta (s-\tau))ds \big|^{\tilde{p}} \Big|\mathcal{F}_{\kappa(t)}\Big )\\
&~~~+\mathbb{E} \Big (\big |\int _{\kappa(t)}^t g_\Delta (\bar{x} _\Delta (s), \bar{x} _\Delta (s-\tau))dB(s)\big |^{\tilde{p}} \Big|\mathcal{F}_{\kappa(t)}\Big )\\
&~~~+\mathbb{E} \Big ( \big|\int _{\kappa(t)}^t h(\bar{x} _\Delta (s^-), \bar{x} _\Delta ((s-\tau)^-))dN(s) \big|^{\tilde{p}} \Big|\mathcal{F}_{\kappa(t)}\Big ) \Big )\\
&\leq c_{\tilde{p}} \Big ( ( \alpha(\Delta))^{\tilde{p}}\Delta ^{\frac{\tilde{p}}{2}} +\mathbb{E} \Big ( \big|\int _{\kappa(t)}^t h(\bar{x} _\Delta (s), \bar{x} _\Delta (s-\tau))dN(s)\big |^{\tilde{p}} \Big|\mathcal{F}_{\kappa(t)}\Big ) \Big ). \\
\end{split}
\end{equation*}
By the characteristic function's argument \cite{6}, for $\Delta \in (0,1]$, we could get
\begin{equation}\label{np1}
\begin{split}
\mathbb{E} |\Delta N_k|^{\tilde{p}} \leq c_0 \Delta,
\end{split}
\end{equation}
where $c_0$ is a positive constant independent of $\Delta$. Then by \eqref{j1} we have
\begin{equation}
\begin{split}
&\mathbb{E} \Big ( \big |\int _{\kappa(t)}^t h(\bar{x} _\Delta (s), \bar{x} _\Delta (s-\tau))dN(s) \big |^{\tilde{p}} \Big|\mathcal{F}_{\kappa(t)}\Big )\\
&=\mathbb{E} \left( | h(x _\Delta (\kappa(t)), x_\Delta (\kappa(t-\tau)))\Delta N_k |^{\tilde{p}} \big|\mathcal{F}_{\kappa(t)}\right)\\
&\leq | h(x _\Delta (\kappa(t)), x_\Delta (\kappa(t-\tau)))|^{\tilde{p}}\mathbb{E} |\Delta N_k|^{\tilde{p}}\\
&\leq c_{\tilde{p}} (1+\mathbb{E}|\bar{x}_\Delta (t)|^{\tilde{p}}+\mathbb{E}|\bar{x}_\Delta (t-\tau)|^{\tilde{p}})\Delta.
\end{split}
\end{equation}
Thus we derive that
\begin{equation*}
\begin{split}
&\mathbb{E} \left ( |x_\Delta (t)-\bar{x}_\Delta (t)|^{\tilde{p}} \big|\mathcal{F}_{\kappa(t)}\right )\\
&\leq c_{\tilde{p}} \left ( ( \alpha(\Delta))^{\tilde{p}}\Delta ^{\frac{\tilde{p}}{2}} +(1+\mathbb{E}|\bar{x}_\Delta (t)|^{\tilde{p}}+\mathbb{E}|\bar{x}_\Delta (t-\tau)|^{\tilde{p}})\Delta \right )\\
&\leq c_{\tilde{p}} (( \alpha(\Delta))^{\tilde{p}}\Delta ^{\frac{\tilde{p}}{2}} +\Delta)(1+\mathbb{E}|\bar{x}_\Delta (t)|^{\tilde{p}}+\mathbb{E}|\bar{x}_\Delta (t-\tau)|^{\tilde{p}}).
\end{split}
\end{equation*}
When $0<\tilde{p}< 2$, an application of Jensen's inequality yields that
\begin{equation*}
\begin{split}
&\mathbb{E} \left ( |x_\Delta (t)-\bar{x}_\Delta (t)|^{\tilde{p}} \big|\mathcal{F}_{\kappa(t)}\right )\\
&\leq \left (\mathbb{E} \left ( |x_\Delta (t)-\bar{x}_\Delta (t)|^{2} \big|\mathcal{F}_{\kappa(t)}\right )\right )^{\frac{\tilde{p}}{2}}\\
&\leq c_{\tilde{p}} (( \alpha(\Delta))^{2}\Delta  +\Delta
)^{\frac{\tilde{p}}{2}}(1+\mathbb{E}|\bar{x}_\Delta (t)|^{2}+\mathbb{E}|\bar{x}_\Delta (t-\tau)|^{2})^{\frac{\tilde{p}}{2}}\\
&\leq c_{\tilde{p}} (\alpha(\Delta))^{\tilde{p}}\Delta ^{\frac{\tilde{p}}{2}} (1+\mathbb{E}|\bar{x}_\Delta (t)|^{\tilde{p}}+\mathbb{E}|\bar{x}_\Delta (t-\tau)|^{\tilde{p}}).
\end{split}
\end{equation*}
We complete the proof.
\end{proof}
\begin{lem}\label{L7}
Let Assumption \ref{a1} and \ref{a3} hold, then,  for any $2 < q < \bar{p} $,  we have
\begin{equation}
\sup _{0<\Delta \leq 1} \sup _{0\leq t \leq T}\mathbb{E} |x_\Delta (t)|^q  \leq C,~~~\forall  T>0.
\end{equation}
\end{lem}
\begin{proof}
For any $\Delta \in \left ( 0,1\right ]$ and $t > 0$, by It\^{o}'s formula and \eqref{y2} we derive that
\begin{equation}\label{A}
\begin{split}
&\mathbb{E} |x_\Delta (t)|^q\\
&\leq \|\xi\|^q + \mathbb{E} \int _0^t q |x_\Delta (s)|^{q-2}(x_\Delta ^T (s)f_\Delta (\bar{x}_\Delta (s),\bar{x}_\Delta (s-\tau))\\
&~~~+\frac{q-1}{2}|g_\Delta (\bar{x}_\Delta (s),\bar{x}_\Delta (s-\tau))|^2)ds\\
&~~~ + \lambda \mathbb{E}\int _0^{t}  (|x_\Delta(s^-)+h(\bar{x}_\Delta(s^-),\bar{x}_\Delta((s-\tau)^-))|^q-|x_\Delta(s^-)|^q)ds\\
&\leq \|\xi\|^q +\mathbb{E} \int _0^t 3q K_3([ 1/\varphi ^{-1}(\alpha(1))]\vee 1 )|x_\Delta (s)|^{q-2}(1+|\bar{x}_\Delta (s)|^2+|\bar{x}_\Delta (s-\tau)|^2)ds\\
&~~~+\mathbb{E} \int _0^t q |x_\Delta (s)|^{q-2}|x_\Delta(s)-\bar{x}_\Delta(s)|| f_\Delta (\bar{x}_\Delta (s),\bar{x}_\Delta (s-\tau))|ds\\
&~~~+\lambda \mathbb{E}\int _0^{t}  (|x_\Delta(s^-)+h(\bar{x}_\Delta(s^-),\bar{x}_\Delta((s-\tau)^-))|^q-|x_\Delta(s^-)|^q)ds\\
&=: \|\xi\|^q +A_1+A_2+A_3.
\end{split}
\end{equation}
We now handle $A_1$, $A_2$, $A_3$ respectively.
Firstly, we could see that
\begin{equation}\label{A1}
\begin{split}
A_1 &\leq
C\mathbb{E} \int _0^t(1+|x_\Delta (s)|^{q}+|\bar{x}_\Delta (s)|^q+|\bar{x}_\Delta (s-\tau)|^q)ds\\
&\leq C\left ( 1+\int _0^t\sup _{0\leq u \leq s}\mathbb{E}|x_\Delta (u)|^{q}ds\right ).
\end{split}
\end{equation}
Let
\begin{flalign}
~~~&A_{21} =
C \mathbb{E} \int _0^t |\bar{x}_\Delta (s)|^{q-2}|x_\Delta(s)-\bar{x}_\Delta(s)|| f_\Delta (\bar{x}_\Delta (s),\bar{x}_\Delta (s-\tau))|ds,&
\end{flalign}
and
\begin{flalign}
~~~&A_{22} =
C \mathbb{E} \int _0^t  |x_\Delta (s)-\bar{x}_\Delta (s)|^{q-2}|x_\Delta(s)-\bar{x}_\Delta(s)|| f_\Delta (\bar{x}_\Delta (s),\bar{x}_\Delta (s-\tau))|ds.&
\end{flalign}
Then we get
\begin{equation*}
\begin{split}
A_{2} \leq A_{21}+A_{22}.
\end{split}
\end{equation*}
By Lemma \ref{L6},  \eqref{y3}, \eqref{y1} and Young's inequality, we have
\begin{equation}\label{as5}
\begin{split}
A_{21}
&\leq C \int _0^t \mathbb{E}\left [|\bar{x}_\Delta (s)|^{q-2}| f_\Delta (\bar{x}_\Delta (s),\bar{x}_\Delta (s-\tau))| \mathbb{E}\big (|x_\Delta(s)-\bar{x}_\Delta(s)|\big|\mathcal{F}_{\kappa(s)}\big )\right ] ds\\
&\leq C \int _0^t (\alpha(\Delta))^2 \Delta ^{\frac{1}{2}}\mathbb{E}\big [|\bar{x}_\Delta (s)|^{q-2}
(1+|\bar{x}_\Delta (s)|+|\bar{x}_\Delta (s-\tau)|) \big ] ds\\
&\leq C \int _0^t (1+\mathbb{E}|\bar{x}_\Delta (s)|^q+\mathbb{E}|\bar{x}_\Delta (s-\tau)|^q) ds\\
&\leq C\left ( 1+\int _0^t\sup _{0\leq u \leq s}\mathbb{E}|x_\Delta (u)|^{q}ds\right ).
\end{split}
\end{equation}
By \eqref{y1}, we obtain that
\begin{equation}\label{as1}
\begin{split}
A_{22} \leq
C \alpha(\Delta) \int _0^t  \mathbb{E}|x_\Delta (s)-\bar{x}_\Delta (s)|^{q-1}ds.
\end{split}
\end{equation}
This, together with \eqref{y3} and Lemma \ref{L6}, implies
\begin{equation}\label{as7}
\begin{split}
A_{22} \leq C\left ( 1+\int _0^t\sup _{0\leq l \leq s}\mathbb{E}|x_\Delta (l)|^{q}ds\right ).
\end{split}
\end{equation}
Combing \eqref{as5} and \eqref{as7} together, we  derive that
\begin{equation}\label{A2}
\begin{split}
A_{2} \leq C\left ( 1+\int _0^t\sup _{0\leq l \leq s}\mathbb{E}|x_\Delta (l)|^{q}ds\right ).
\end{split}
\end{equation}
  By \eqref{j1} one can see that
\begin{equation}\label{A3}
\begin{split}
A_3 &\leq
C\mathbb{E} \int _0^t(1+|x_\Delta (s)|^{q}+|\bar{x}_\Delta (s)|^q+|\bar{x}_\Delta (s-\tau)|^q)ds\\
&\leq C\left ( 1+\int _0^t\sup _{0\leq u \leq s}\mathbb{E}|x_\Delta (u)|^{q}ds\right ).
\end{split}
\end{equation}
Substituting (\ref{A1}), (\ref{A2}) and (\ref{A3}) into (\ref{A}), we obtain
\begin{equation}
\begin{split}
\sup _{0\leq u \leq t}\mathbb{E}|x_\Delta (u)|^{q}\leq C\left ( 1+\int _0^t\sup _{0\leq u \leq s}\mathbb{E}|x_\Delta (u)|^{q}ds\right ).
\end{split}
\end{equation}
An application of Gronwall's inequality yields that
\begin{equation}
\begin{split}
\sup _{0\leq u \leq T}\mathbb{E}|x_\Delta (u)|^{q}\leq C,
\end{split}
\end{equation}
where $C$ is independent of $\Delta$. Since  this inequality holds for any $\Delta\in (0,1]$, the desired result follows.  The proof is therefore complete.
\end{proof}
\begin{lem}\label{L8}
Let Assumption \ref{a1} and \ref{a3} hold. Then for any $\Delta \in (0,1] $ and $t\in [0,T]$, we have
\begin{equation}\label{bs1}
\begin{split}
\mathbb{E}  |x_\Delta (t)-\bar{x}_\Delta (t)|^{q} \leq C \left((\alpha(\Delta))^{q}\Delta ^{\frac{q}{2}} +\Delta\right),~~~2\leq q < \bar{p},
\end{split}
\end{equation}
and
\begin{equation}\label{bs2}
\begin{split}
\mathbb{E} |x_\Delta (t)-\bar{x}_\Delta (t)|^{q}\leq C (\alpha(\Delta))^{q}\Delta ^{\frac{q}{2}} ,~~~0<q< 2.
\end{split}
\end{equation}
Hence,
\begin{equation}\label{bs3}
\lim _{\Delta \rightarrow 0} \mathbb{E} |x_\Delta (t)-\bar{x}_\Delta (t)|^q =0,~~~q>0.
\end{equation}
\end{lem}
\begin{proof}
By Lemma \ref{L7} and (\ref{3.26}), we get (\ref{bs1}) when $2\leq q < \bar{p}$. Then for any $0<q< 2$, Jensen's inequality gives that
\begin{equation}
\begin{split}
&\mathbb{E} |x_\Delta (t)-\bar{x}_\Delta (t)|^{q}
\leq \left ( \mathbb{E} |x_\Delta (t)-\bar{x}_\Delta (t)|^{2} \right )^{\frac{q}{2}} \leq C \left((\alpha(\Delta))^{2}\Delta  +\Delta \right )^{\frac{q}{2}}
\leq C (\alpha(\Delta))^{q}\Delta ^{\frac{q}{2}}.
\end{split}
\end{equation}
We complete the proof.
\end{proof}
By using the Chebyshev inequality, Lemma \ref{L4} and Lemma \ref{L7}, we can immediately have the following lemma.
\begin{lem}\label{L9}
Let Assumption \ref{a1} and \ref{a3} hold. For any real number $L> \|\xi\|$, define the stopping time
\begin{equation}\label{bs4}
\tau _L= \inf \{ t\geq 0:|x(t)| \geq L \}\,  \mbox{ and } \, \tau _{\Delta,L}= \inf \{ t\geq 0:|x_\Delta (t)| \geq L \}.
\end{equation}
 Then we have
\begin{equation}\label{bs5}
\begin{split}
\mathbb{P}(\tau _L \leq T)\leq \frac{C}{L^2}\,  \mbox{ and } \, \mathbb{P}(\tau _{\Delta,L} \leq T)\leq \frac{C}{L^2}.
\end{split}
\end{equation}
\end{lem}

Let us now discuss the rate of convergence  in $\mathcal{L}^p$($p \geq 2$) sense for the truncated EM solutions to the true solution. In order to avoid the details becoming more complicated, we only calculate the rate of convergence when $p=2$. 
\begin{thm}\label{thm1}
Let Assumption \ref{2.1}, \ref{a1}--\ref{a3} hold. Suppose that there exists a real number $q\in (2,\bar{p})$ such that
$
q>(1+\beta)\bar{\eta}.
$
Then,  for any $\Delta\in(0,1]$, we have
\begin{equation}\label{cs3}
\mathbb{E}|x(T)-x_{\Delta}(T)|^2 \leq C\left ( (\varphi ^{-1}(\alpha(\Delta)))^{2\beta +2 -q} +(\alpha(\Delta))^2 \Delta + \Delta^{\frac{q-2\beta}{q}}+ \Delta ^{2\gamma}
\right ),
\end{equation}
and
\begin{equation}\label{cs4}
\mathbb{E}|x(T)-\bar{x}_{\Delta}(T)|^2 \leq C\left ( (\varphi ^{-1}(\alpha(\Delta)))^{2\beta +2 -q} +(\alpha(\Delta))^2 \Delta + \Delta^{\frac{q-2\beta}{q}}+ \Delta ^{2\gamma}
\right ).
\end{equation}
In particular, let
\begin{equation}\label{cs5}
\varphi (r)=c^* r^{1+\beta},~\forall r\geq 1, \, \alpha(\Delta) =K_0 \Delta ^{-\varepsilon}\, \varepsilon \in (0,1/4],
\end{equation}
with $c^* =4K_1 +|f(0,0)|+|g(0,0)|$.
 Then it holds  for any $\Delta\in(0,1]$ that
\begin{equation}\label{cs7}
\mathbb{E}|x(T)-x_{\Delta}(T)|^2 \leq C\Delta^{[\frac{\varepsilon(q-2\beta -2)}{1+\beta}]\wedge[1-2\varepsilon]\wedge [\frac{q-2\beta}{q}]\wedge [2\gamma]
},
\end{equation}
and
\begin{equation}\label{cs8}
\mathbb{E}|x(T)-\bar{x}_{\Delta}(T)|^2 \leq C\Delta^{[\frac{\varepsilon(q-2\beta -2)}{1+\beta}]\wedge[1-2\varepsilon]\wedge [\frac{q-2\beta}{q}]\wedge [2\gamma]
}.
\end{equation}
\end{thm}
\begin{proof}
Let $e_\Delta(t)=x(t)-x_\Delta(t)$ for $t\geq 0$ and $\Delta\in(0,1]$. Define $\rho _{\Delta,L} = \tau _L \wedge \tau _{\Delta,L}$, that is
\begin{equation*}
\rho _{\Delta,L}= \inf \{ t\geq 0:|x(t)|\vee|x_\Delta (t)| \geq L \}.
\end{equation*}
We write $\rho _{\Delta,L}=\rho$ for simplicity. Noting that for $\bar{q}\in (2,\bar{\eta})$, we have $q>(1+\beta)\bar{q}$.
By It\^{o}'s formula, for any $0\leq t \leq T$ and $\Delta\in(0,1]$, we have
\begin{equation}\label{cs1}
\begin{split}
&\mathbb{E} |e_\Delta(t \wedge \rho)|^2 \\
&\leq \mathbb{E} \int _0^{t \wedge \rho} 2\Big (e_\Delta ^T (s)(f(x(s),x(s-\tau))-f_\Delta (\bar{x}_\Delta (s),\bar{x}_\Delta (s-\tau))) \\
&~~~ + \frac{1}{2}|g(x(s),x(s-\tau))- g_\Delta (\bar{x}_\Delta (s),\bar{x}_\Delta (s-\tau))|^2 \Big )ds\\
&~~~ + \lambda \mathbb{E}\int _0^{t\wedge \rho} \big (
|e_\Delta (s)+h(x(s^-),x((s-\tau)^-))\\
&~~~-h(\bar{x}_\Delta(s^-),\bar{x}_\Delta((s-\tau)^-))|^2
-|e_\Delta (s)|^2 \big )ds\\
&=: I_1 +I_2 .
\end{split}
\end{equation}
Firstly, we estimate $I_1$. Noting that
\begin{equation*}
\begin{split}
&\frac {1}{2} |g(x(s),x(s-\tau))- g_\Delta (\bar{x}_\Delta (s),\bar{x}_\Delta (s-\tau))|^2 \\
& \leq \frac {1}{2} \Big ( (\bar{q}-1)|g(x(s),x(s-\tau))- g(x_\Delta (s),x_\Delta (s-\tau))|^2\\
&~~~+ \frac{\bar{q}-1}{\bar{q}-2} |g(x_\Delta(s),x_\Delta (s-\tau))- g_\Delta(\bar{x}_\Delta (s),\bar{x}_\Delta (s-\tau))|^2 \Big )\\
&=\frac {\bar{q} -1}{2} |g(x(s),x(s-\tau))- g(x_\Delta (s),x_\Delta (s-\tau))|^2\\
&~~~+ \frac{\bar{q}-1}{2(\bar{q}-2) } |g(x_\Delta(s),x_\Delta (s-\tau))- g_\Delta(\bar{x}_\Delta (s),\bar{x}_\Delta (s-\tau))|^2. \\
\end{split}
\end{equation*}
Then we have
\begin{equation}
\begin{split}
I_1
&\leq \mathbb{E} \int _0^{t \wedge \rho} 2\Big (e_\Delta ^T (s)(f(x(s),x(s-\tau))-f(x_\Delta(s),x_\Delta (s-\tau))) \\
&~~~ + \frac {\bar{q} -1}{2} |g(x(s),x(s-\tau))- g(x_\Delta (s),x_\Delta (s-\tau))|^2 \Big )ds\\
&~~~ +\mathbb{E} \int _0^{t \wedge \rho} 2\Big (e_\Delta ^T (s)(f(x_\Delta(s),x_\Delta (s-\tau))-f_\Delta(\bar{x}_\Delta (s),\bar{x}_\Delta (s-\tau))) \\
&~~~+\frac{\bar{q}-1}{2(\bar{q}-2) } |g(x_\Delta(s),x_\Delta (s-\tau))- g_\Delta(\bar{x}_\Delta (s),\bar{x}_\Delta (s-\tau))|^2\Big )ds\\
&=:I_{11}+I_{12}.
\end{split}
\end{equation}
By Assumptions \ref{2.1}, \ref{a2} and \eqref{zs1}, we derive that
\begin{equation}
\begin{split}
I_{11}
&\leq \mathbb{E} \int _0^{t \wedge \rho} 2\Big [K_2 (|x(s)-x_\Delta (s)|^2+|x(s-\tau)-x_\Delta (s-\tau)|^2)\\
&~~~-U(x(s), x_\Delta (s))+U(x(s-\tau), x_\Delta (s-\tau)) \Big ]ds\\
&\leq 4 K_2 \mathbb{E} \int _0^{t \wedge \rho} |e_\Delta(s)|^2 ds +2K_2 \int _{-\tau}^{0} |\xi(s)-\xi(\kappa(s))|^2 ds\\
&~~~+2\mathbb{E} \int _0^{t \wedge \rho} \Big (-U(x(s), x_\Delta (s))+U(x(s-\tau), x_\Delta (s-\tau)) \Big )ds\\
&\leq 4 K_2 \int _0^{t} \mathbb{E}|e_\Delta(s\wedge \rho)|^2 ds +2\tau K_2 \bar{K}^2 \Delta^{2\gamma} +2\int _{-\tau}^{0} U(\xi(s), \xi(\kappa(s)))ds\\
&\leq 4 K_2 \int _0^{t} \mathbb{E}|e_\Delta(s\wedge \rho)|^2 ds +2\tau K_2 \bar{K}^2 \Delta^{2\gamma} +2\int _{-\tau}^{0} \kappa_b|\xi(s)-\xi(\kappa(s))|^2ds\\
&\leq 4 K_2 \int _0^{t} \mathbb{E}|e_\Delta(s\wedge \rho)|^2 ds +c_1 \Delta^{2\gamma},
\end{split}
\end{equation}
where $c_1=2\tau K_2 \bar{K}^2+2\tau \kappa_b \bar{K}^2$, $\kappa(s)$ is defined as before.
As for $I_{12}$, we have
\begin{equation}
\begin{split}
I_{12}
&\leq \mathbb{E} \int _0^{t \wedge \rho} 2\Big ( e_\Delta ^T (s)(f(x_\Delta(s),x_\Delta (s-\tau))-f_\Delta(x_\Delta (s),x_\Delta (s-\tau))) \\
&~~~+\frac{\bar{q}-1}{\bar{q}-2} |g(x_\Delta(s),x_\Delta (s-\tau))- g_\Delta(x_\Delta (s),x_\Delta (s-\tau))|^2\Big )ds\\
&~~~+ \mathbb{E} \int _0^{t \wedge \rho} 2\Big ( e_\Delta ^T (s)(f_\Delta(x_\Delta(s),x_\Delta (s-\tau))-f_\Delta(\bar{x}_\Delta (s),\bar{x}_\Delta (s-\tau))) \\
&~~~+\frac{\bar{q}-1}{\bar{q}-2} |g_\Delta(x_\Delta(s),x_\Delta (s-\tau))- g_\Delta(\bar{x}_\Delta (s),\bar{x}_\Delta (s-\tau))|^2\Big )ds\\
&=:I_{121}+I_{122} .
\end{split}
\end{equation}
By Young's inequality, Assumption \ref{a1} and \eqref{zs2}, we have
\begin{equation*}
\begin{split}
I_{121}
&\leq \mathbb{E} \int _0^{t \wedge \rho} \Big (|e_\Delta (s)|^2+|f(x_\Delta(s),x_\Delta (s-\tau))-f_\Delta(x_\Delta (s),x_\Delta (s-\tau))|^2 \\
&~~~+\frac{2(\bar{q}-1)}{\bar{q}-2} |g(x_\Delta(s),x_\Delta (s-\tau))- g_\Delta(x_\Delta (s),x_\Delta (s-\tau))|^2\Big )ds\\
&\leq \mathbb{E} \int _0^{t \wedge \rho} |e_\Delta (s)|^2 ds\\
&~~~+C\mathbb{E} \int _0^{T} \Big (|f(x_\Delta(s),x_\Delta (s-\tau))-f_\Delta(x_\Delta (s),x_\Delta (s-\tau))|^2 \\
&~~~+|g(x_\Delta(s),x_\Delta (s-\tau))- g_\Delta(x_\Delta (s),x_\Delta (s-\tau))|^2\Big )ds\\
&\leq \int _0^{t}\mathbb{E} |e_\Delta (s \wedge \rho)|^2 ds+C\mathbb{E} \int _0^{T} \Big [(1+|x_\Delta (s)|^{2\beta}+|x_\Delta (s-\tau)|^{2\beta}\\
&~~~+|\pi_\Delta (x_\Delta (s))|^{2\beta}+|\pi_\Delta(x_\Delta (s-\tau))|^{2\beta} )\\
&~~~(|x_\Delta (s)-\pi_\Delta (x_\Delta (s))|^{2}+|x_\Delta (s-\tau)-\pi_\Delta(x_\Delta (s-\tau))|^2)\Big ]ds\\
&\leq \int _0^{t}\mathbb{E} |e_\Delta (s \wedge \rho)|^2 ds+C \int _0^{T} \mathbb{E}\Big [(1+|x_\Delta (s)|^{2\beta}+|x_\Delta (s-\tau)|^{2\beta} )\\
&~~~(|x_\Delta (s)-\pi_\Delta (x_\Delta (s))|^{2}+|x_\Delta (s-\tau)-\pi_\Delta(x_\Delta (s-\tau))|^2)\Big ]ds. \\
\end{split}
\end{equation*}
Then by Chebyshev's inequality, H\"{o}lder's inequality and Lemma \ref{L7}, we get
\begin{equation}
\begin{split}
I_{121}
&\leq \int _0^{t}\mathbb{E} |e_\Delta (s \wedge \rho)|^2 ds+C \int _0^{T} \Big [\mathbb{E}(1+|x_\Delta (s)|^{q}+|x_\Delta (s-\tau)|^{q} )\Big ]^{\frac{2\beta}{q}}\\
&~~~\Big [\mathbb{E}(|x_\Delta (s)-\pi_\Delta (x_\Delta (s))|^{\frac{2q}{q-2\beta}}+|x_\Delta (s-\tau)-\pi_\Delta(x_\Delta (s-\tau))|^{\frac{2q}{q-2\beta}})\Big ]^{\frac{q-2\beta}{q}}ds\\
&\leq \int _0^{t}\mathbb{E} |e_\Delta (s \wedge \rho)|^2 ds+C \int _0^{T} \Big [\mathbb{E}|x_\Delta (s)-\pi_\Delta (x_\Delta (s))|^{\frac{2q}{q-2\beta}}\Big ]^{\frac{q-2\beta}{q}} ds\\
&~~~+C \int _{-\tau}^{0} \Big [\mathbb{E}|\xi(\kappa(s))-\pi_\Delta(\xi(\kappa(s)))|^{\frac{2q}{q-2\beta}}\Big ]^{\frac{q-2\beta}{q}} ds\\
&\leq \int _0^{t} \mathbb{E}|e_\Delta (s \wedge \rho)|^2 ds+C \int _0^{T} \Big [\mathbb{E}\left (\textbf{I}_{|x_\Delta (s)|>\varphi^{-1}(\alpha(\Delta))}|x_\Delta (s)|^{\frac{2q}{q-2\beta}}\right ) \Big ]^{\frac{q-2\beta}{q}} ds\\
&~~~+C \int _{-\tau}^{0} \Big [\mathbb{E}\left ( \textbf{I}_{|\xi(\kappa(s))|>\varphi^{-1}(\alpha(\Delta))}|\xi(\kappa(s))|^
{\frac{2q}{q-2\beta}}\right ) \Big ]^{\frac{q-2\beta}{q}} ds\\
&\leq \int _0^{t}\mathbb{E} |e_\Delta (s \wedge \rho)|^2 ds\\
&~~~+C \int _0^{T}\Big ( \left [ \mathbb{P}(|x_\Delta (s)|>\varphi^{-1}(\alpha(\Delta)))\right ]^{\frac{q-2\beta -2}{q-2\beta}} \left [ \mathbb{E}|x_\Delta (s)|^q \right ]^{\frac{2}{q-2\beta}}
\Big )^{\frac{q-2\beta}{q}} ds\\
&~~~+C \int _{-\tau}^{0}\Big ( \left [ \mathbb{P}(|\xi(\kappa(s))|>\varphi^{-1}(\alpha(\Delta)))\right ]^{\frac{q-2\beta -2}{q-2\beta}} \left [ \mathbb{E}|\xi(\kappa(s))|^q \right ]^{\frac{2}{q-2\beta}}
\Big )^{\frac{q-2\beta}{q}} ds\\
&\leq \int _0^{t}\mathbb{E} |e_\Delta (s \wedge \rho)|^2 ds+C \int _0^{T}\Big ( \frac{\mathbb{E}|x_\Delta (s)|^q}{(\varphi^{-1}(\alpha(\Delta)))^q}
\Big )^{\frac{q-2\beta-2}{q}} ds\\
&~~~+C \int _{-\tau}^{0}\Big ( \frac{\mathbb{E}|\xi(\kappa(s))|^q}{(\varphi^{-1}(\alpha(\Delta)))^q}
\Big )^{\frac{q-2\beta-2}{q}} ds\\
&\leq \int _0^{t} \mathbb{E}|e_\Delta (s \wedge \rho)|^2 ds+C (\varphi^{-1}(\alpha(\Delta)))^{2\beta +2 -q} .
\end{split}
\end{equation}

We can use the same technique to handle $I_{122}$. By Young's inequality, H\"{o}lder's inequality, Lemma \ref{L7}, Lemma \ref{L8}, \eqref{zs3} and the inequality $2q/(q-2\beta) \geq 2$, we obtain
\begin{equation}
\begin{split}
I_{122}
&\leq \mathbb{E} \int _0^{t \wedge \rho} \Big (|e_\Delta (s)|^2 +|f_\Delta(x_\Delta(s),x_\Delta (s-\tau))-f_\Delta(\bar{x}_\Delta (s),\bar{x}_\Delta (s-\tau))|^2 \\
&~~~+\frac{2(\bar{q}-1)}{\bar{q}-2} |g_\Delta(x_\Delta(s),x_\Delta (s-\tau))- g_\Delta(\bar{x}_\Delta (s),\bar{x}_\Delta (s-\tau))|^2\Big )ds\\
&\leq \mathbb{E} \int _0^{t \wedge \rho} |e_\Delta (s)|^2 ds\\
&~~~+C\mathbb{E} \int _0^{T} \Big (|f_\Delta(x_\Delta(s),x_\Delta (s-\tau))-f_\Delta(\bar{x}_\Delta (s),\bar{x}_\Delta (s-\tau))|^2 \\
&~~~+|g_\Delta(x_\Delta(s),x_\Delta (s-\tau))- g_\Delta(\bar{x}_\Delta (s),\bar{x}_\Delta (s-\tau))|^2\Big )ds\\
&\leq \int _0^{t}\mathbb{E} |e_\Delta (s \wedge \rho)|^2 ds\\
&~~~+C \int _0^{T} \mathbb{E}\Big [(1+|x_\Delta (s)|^{2\beta}+|x_\Delta (s-\tau)|^{2\beta}+|\bar{x}_\Delta (s)|^{2\beta}+|\bar{x}_\Delta (s-\tau)|^{2\beta} )\\
&~~~(|x_\Delta (s)-\bar{x}_\Delta (s)|^{2}+|x_\Delta (s-\tau)-\bar{x}_\Delta (s-\tau)|^2)\Big ]ds\\
&\leq \int _0^{t}\mathbb{E} |e_\Delta (s \wedge \rho)|^2 ds\\
&~~~+C \int _0^{T}\big [\mathbb{E}(1+|x_\Delta (s)|^{q}+|x_\Delta (s-\tau)|^{q}+|\bar{x}_\Delta (s)|^{q}+|\bar{x}_\Delta (s-\tau)|^{q} )\big ]^{\frac{2\beta}{q}}\\
&~~~\big [\mathbb{E}(|x_\Delta (s)-\bar{x}_\Delta (s)|^{\frac{2q}{q-2\beta}}+|x_\Delta (s-\tau)-\bar{x}_\Delta (s-\tau)|^{\frac{2q}{q-2\beta}})\big ]^{\frac{q- 2\beta}{q}}ds\\
&\leq \int _0^{t}\mathbb{E} |e_\Delta (s \wedge \rho)|^2 ds+C \int _0^{T}\big (\mathbb{E}|x_\Delta (s)-\bar{x}_\Delta (s)|^{\frac{2q}{q-2\beta}}\big )^{\frac{q- 2\beta}{q}}ds\\
&\leq \int _0^{t}\mathbb{E} |e_\Delta (s \wedge \rho)|^2 ds+C \int _0^{T}\big ( (\alpha(\Delta))^{\frac{2q}{q-2\beta}} \Delta ^{\frac{q}{q-2\beta}} + \Delta \big )^{\frac{q- 2\beta}{q}}ds\\
&\leq \int _0^{t}\mathbb{E} |e_\Delta (s \wedge \rho)|^2 ds+C\left ( (\alpha(\Delta))^2 \Delta + \Delta^{\frac{q-2\beta}{q}} \right ).
\end{split}
\end{equation}
Let us now estimate $I_2$. By Assumption \ref{a1}, we obtain
\begin{equation*}
\begin{split}
&~~~|x(s^-)-x_{\Delta}(s)+h(x(s^-),x((s-\tau)^-))-h(\bar{x}_\Delta(s),\bar{x}_\Delta(s-\tau))|^2\\
&\leq 2\big ( |x(s^-)-x_{\Delta}(s)|^2+|h(x(s^-),x((s-\tau)^-))-h(\bar{x}_\Delta(s),\bar{x}_\Delta(s-\tau))|^2\big )\\
&\leq 2\big ( |x(s^-)-x_{\Delta}(s)|^2 +2K_1^2(|x(s^-)-\bar{x}_{\Delta}(s)|^2+|x((s-\tau)^-)-\bar{x}_\Delta(s-\tau)|^2) \big )\\
&\leq 2\big ( |x(s^-)-x_{\Delta}(s)|^2 +4K_1^2(|x(s^-)-x_{\Delta}(s)|^2+|x_{\Delta}(s)-\bar{x}_{\Delta}(s)|^2\\
&~~~+|x((s-\tau)^-)-x_\Delta(s-\tau)|^2+|x_\Delta(s-\tau)-\bar{x}_\Delta(s-\tau)|^2) \big ). \\
\end{split}
\end{equation*}
Thus, using Assumption \ref{2.1} and Lemma \ref{L8}, we obtain
\begin{equation}\label{cs2}
\begin{split}
I_2&\leq \lambda \mathbb{E}\int _0^{t\wedge \rho}\big ( 2|x(s^-)-x_{\Delta}(s)|^2 +8K_1^2(|x(s^-)-x_{\Delta}(s)|^2+|x_{\Delta}(s)-\bar{x}_{\Delta}(s)|^2\\
&~~~~~~+|x((s-\tau)^-)-x_\Delta(s-\tau)|^2+|x_\Delta(s-\tau)-\bar{x}_\Delta(s-\tau)|^2)\\
&~~~~~~-|x(s^-)-x_{\Delta}(s)|^2 \big )ds\\
&\leq \lambda \mathbb{E}\int _0^{t\wedge \rho}\big ( 2|e_\Delta (s)|^2 + 16K_1^2(|e_\Delta (s)|^2+|x_{\Delta}(s)-\bar{x}_{\Delta}(s)|^2) -|e_\Delta (s)|^2 \big )ds\\
&~~~+\lambda \int _{-\tau}^0 8K_1^2|\xi(s)-\xi(\kappa(s))|^2 ds\\
&\leq \lambda(16K_1^2 +1)\mathbb{E}\int _0^{t\wedge \rho} |e_\Delta (s)|^2 ds +16\lambda K_1^2 \int _0^{T} \mathbb{E} |x_{\Delta}(s)-\bar{x}_{\Delta}(s)|^2 ds \\
&~~~+ 8\lambda\tau K_1^2 \bar{K}^2 \Delta ^{2\gamma}\\
&\leq \lambda(16K_1^2 +1)\int _0^{t} \mathbb{E}|e_\Delta (s\wedge \rho)|^2 ds +16\lambda K_1^2 T (\alpha(\Delta))^2 \Delta+ 8\lambda\tau K_1^2 \bar{K}^2 \Delta ^{2\gamma}.
\end{split}
\end{equation}
Combining \eqref{cs1} - \eqref{cs2} together, one can see that
\begin{equation*}
\begin{split}
&\mathbb{E}|e_\Delta (t \wedge \rho)|^2\\
&\leq C\left (\int _0^{t} \mathbb{E}|e_\Delta (s \wedge \rho)|^2 ds+ (\varphi^{-1}(\alpha(\Delta)))^{2\beta +2 -q} + (\alpha(\Delta))^2 \Delta + \Delta^{\frac{q-2\beta}{q}} + \Delta^{2\gamma}\right ).
\end{split}
\end{equation*}
An application of Gronwall's inequality yields that
\begin{equation*}
\begin{split}
\mathbb{E}|e_\Delta (T \wedge \rho)|^2
&\leq C\left ( (\varphi^{-1}(\alpha(\Delta)))^{2\beta +2 -q} + (\alpha(\Delta))^2 \Delta + \Delta^{\frac{q-2\beta}{q}} + \Delta^{2\gamma}\right ).
\end{split}
\end{equation*}
The desired results \eqref{cs3} and \eqref{cs4} follow by letting $L \rightarrow \infty$ and using Lemma \ref{L9}  and Lemma \ref{L8}.
In particular, by the definition of $varphi$, we can derive  \eqref{cs7} and \eqref{cs8}.   The proof is therefore complete.
\end{proof}

\section{Convergence  in $\mathcal{L}^p$ for $0< p < 2$}

In this section, we will discuss the convergence and the rate of the convergence  of the truncated EM method for \eqref{dx1} in $\mathcal{L}^p$ for $0<p <2$. To achieve this goal, we need to impose the following assumptions on coefficients.
\begin{ass}\label{a4.1}
There exists a positive constant $K_R$ such that
\begin{equation}
\begin{split}
|f(x,y)-f(\bar{x},\bar{y})|\vee |g(x,y)-g(\bar{x},\bar{y})|\vee |h(x,y)-h(\bar{x},\bar{y})| \leq K_R(|x-\bar{x}|+|y-\bar{y}|)
\end{split}
\end{equation}
for any $x,y,\bar{x},\bar{y} \in \mathbb{R} ^n $ with $|x|\vee |y|\vee|\bar{x}|\vee|\bar{y}|\leq R$.
\end{ass}
\begin{ass}\label{a4.2}
There exist constants $K_5 >0$, $K_6 \geq 0$ and $\sigma > 2$ such that
\begin{equation}
\begin{split}
&2x^T f(x,y) +|g(x,y)|^2 +\lambda (2x^T h(x,y) +|h(x,y)|^2)\\
&\leq K_5(1+|x|^2 +|y|^2)-K_6|x|^\sigma+K_6|y|^\sigma
\end{split}
\end{equation}
for any $x,y \in \mathbb{R} ^n $.
\end{ass}
We could get the following lemma in the similar way as Lemma \ref{L4} was proved.
\begin{lem}\label{BL3}
Let Assumption \ref{a4.1} and \ref{a4.2} hold. Then SDDEwPJs \eqref{dx1} has a unique global solution $x(t)$ which satisfies
\begin{equation}
\sup_{0\leq t\leq T} \mathbb{E}|x(t)|^2 <\infty,~~~\forall T>0.
\end{equation}
\end{lem}

In the previous section, the jump coefficient $h$ is linear growth, but in the assumptions \ref{a4.1} and \ref{a4.2} $h$ is  allowed to grow super-linearly. Thus, we need to truncate all the three coefficients. In the same way in Section 3,   we first choose a strictly increasing continuous function $\varphi (r) :\mathbb{R} _+ \rightarrow \mathbb{R}_+$ such that $\varphi (r) \to \infty$ as $r \rightarrow \infty$ and
\begin{equation}
\sup_{|x|\vee|y| \le r} \left ( |f(x,y)| \vee |g(x,y)| \vee |h(x,y)|\right ) \le \varphi (r),~~~\forall r\geq 1.
\end{equation}
Choose $K_0$ and  $\alpha: (0,1] \rightarrow (0,\infty)$ as in \eqref{y3}.
For a given step size $\Delta \in (0,1]$, the truncated mapping $\pi _{\Delta}$ is defined as \eqref{zs6}, and  the truncated functions are define as follows:
\begin{equation*}\label{es2}
\begin{split}
&f_\Delta (x,y)=f(\pi _{\Delta}(x),\pi _{\Delta}(y)),~~g_\Delta (x,y)=g(\pi _{\Delta}(x),\pi _{\Delta}(y)), ~~
h_\Delta (x,y)=h(\pi _{\Delta}(x),\pi _{\Delta}(y)),
\end{split}
\end{equation*}
for any $x,y \in \mathbb{R} ^n$.
It is easy to see that
\begin{equation}\label{es3}
|f_\Delta (x,y)|\vee|g_\Delta (x,y)|\vee|h_\Delta (x,y)|\leq\varphi(\varphi ^ {-1} (\alpha(\Delta)))\leq \alpha(\Delta),~~~ \forall x,y \in \mathbb{R} ^n.
\end{equation}

Moreover, if Assumption \ref{a4.2} holds, then it holds for any $\Delta \in (0,1], x,y \in \mathbb{R} ^n$ that
\begin{equation}\label{BL4}
\begin{split}
&2x^T f_\Delta(x,y) +|g_\Delta(x,y)|^2 +\lambda (2x^T h_\Delta(x,y) +|h_\Delta(x,y)|^2)\\
&\leq 3K_5([ 1/\varphi ^{-1}(\alpha(1))]\vee 1 )(1+|x|^2 +|y|^2)
-K_6|\pi_\Delta(x)|^\sigma+K_6|\pi_\Delta(y)|^\sigma.
\end{split}
\end{equation}

Since $|h_\Delta (x,y)|\leq \alpha(\Delta)$, similar to the Lemma \ref{L6}, we have the following lemma.
\begin{lem}\label{BL5}
For any $\Delta \in (0,1] $ and $t\in [0,\infty]$, we have
\begin{equation}
\begin{split}
\mathbb{E} |x_\Delta (t)-\bar{x}_\Delta (t)|^{\tilde{p}}\leq c_{\tilde{p}} (\alpha(\Delta))^{\tilde{p}}\Delta ,~~~\tilde{p}\geq 2,
\end{split}
\end{equation}
and
\begin{equation}
\begin{split}
\mathbb{E}  |x_\Delta (t)-\bar{x}_\Delta (t)|^{\tilde{p}} \leq c_{\tilde{p}} (\alpha(\Delta))^{\tilde{p}}\Delta ^{\frac{\tilde{p}}{2}},~~~0<\tilde{p}< 2,
\end{split}
\end{equation}
where $c_{\tilde{p}}$ is a positive constant which is independent of $\Delta$. As a result,
\begin{equation}
\begin{split}
\lim_{\Delta \rightarrow 0}\mathbb{E}  |x_\Delta (t)-\bar{x}_\Delta (t)|^{\tilde{p}} =0,~~~\tilde{p}>0.
\end{split}
\end{equation}
\end{lem}

The following lemma states that the numerical solution is bounded in mean square.
\begin{lem}\label{BL6}
Let Assumption \ref{a4.1} and \ref{a4.2} hold, then we have
\begin{equation}
\sup _{0<\Delta \leq 1} \sup _{0\leq t \leq T}\mathbb{E} |x_\Delta (t)|^2  \leq C,~~~\forall  T>0.
\end{equation}
\end{lem}
\begin{proof} Since the proof is similar to that of Lemma \ref{L7}, we only highlight how to deal with the jump term.  By It\^{o}'s formula and (\ref{BL4}), we derive that, for any $\Delta \in \left ( 0,1\right ]$ and $t\in [0,T]$,
\begin{equation}\label{es8}
\begin{split}
&~~\mathbb{E} |x_\Delta (t)|^2\\
&\leq \|\xi\|^2 + \mathbb{E} \int _0^t \left (2\bar{x}_\Delta ^T (s)f_\Delta (\bar{x}_\Delta (s),\bar{x}_\Delta (s-\tau)) +|g_\Delta (\bar{x}_\Delta (s),\bar{x}_\Delta (s-\tau))|^2\right )ds\\
&~~~ + \lambda \mathbb{E}\int _0^{t}
\left (2\bar{x}_\Delta ^T (s)h_\Delta (\bar{x}_\Delta(s^-),\bar{x}_\Delta((s-\tau)^-)) +|h_\Delta (\bar{x}_\Delta(s^-),\bar{x}_\Delta((s-\tau)^-))|^2
\right )ds\\
&~~~+\mathbb{E} \int _0^t \big ( 2(x_\Delta(s)-\bar{x}_\Delta(s)) ^T f_\Delta (\bar{x}_\Delta (s),\bar{x}_\Delta (s-\tau))\\
&~~~+2\lambda(x_\Delta(s)-\bar{x}_\Delta(s)) ^T h_\Delta (\bar{x}_\Delta(s^-),\bar{x}_\Delta((s-\tau)^-))
\big )ds\\
&\leq \|\xi\|^2 + \mathbb{E} \int _0^t \big (3K_5([ 1/\varphi ^{-1}(\alpha(1))]\vee 1 )(1+|\bar{x}_\Delta (s)|^2 +|\bar{x}_\Delta (s-\tau)|^2)\\
&~~~-K_6|\pi_\Delta(\bar{x}_\Delta (s))|^\sigma+K_6|\pi_\Delta(\bar{x}_\Delta (s-\tau))|^\sigma\big )ds\\
&~~~+2(\lambda+1)\mathbb{E} \int _0^t  (x_\Delta(s)-\bar{x}_\Delta(s)) ^T\big ( f_\Delta (\bar{x}_\Delta (s),\bar{x}_\Delta (s-\tau))\\
&~~~+h_\Delta (\bar{x}_\Delta(s^-),\bar{x}_\Delta((s-\tau)^-))\big )ds. \\
\end{split}
\end{equation}
Moreover, by  \eqref{y3} and Lemma \ref{BL5}, we have
\begin{equation}\label{es10}
\begin{split}
&\mathbb{E} \int _0^t  (x_\Delta(s)-\bar{x}_\Delta(s)) ^T \left ( f_\Delta (\bar{x}_\Delta (s),\bar{x}_\Delta (s-\tau))+h_\Delta (\bar{x}_\Delta(s^-),\bar{x}_\Delta((s-\tau)^-))\right )ds\\
&\leq \mathbb{E} \int _0^t  |x_\Delta(s)-\bar{x}_\Delta(s)| \left ( |f_\Delta (\bar{x}_\Delta (s),\bar{x}_\Delta (s-\tau))|+|h_\Delta (\bar{x}_\Delta(s^-),\bar{x}_\Delta((s-\tau)^-))|\right )ds\\
&\leq 2\alpha (\Delta) \int _0^t  \mathbb{E} |x_\Delta(s)-\bar{x}_\Delta(s)| ds \leq 2T (\alpha (\Delta))^2 \Delta ^{\frac{1}{2}} \leq 2T K_0 ^2.
\end{split}
\end{equation}
Therefore, we have
\begin{equation*}
\begin{split}
\mathbb{E} |x_\Delta (t)|^2
&\leq C\left ( 1+ \int _0^t  (1+\mathbb{E}|\bar{x}_\Delta (s)|^2 +\mathbb{E}|\bar{x}_\Delta (s-\tau)|^2)ds\right )\\
&\leq C\left ( 1+ \int _0^t \sup_{0\leq u \leq s}\mathbb{E}|x_\Delta (u)|^2 ds\right ). \\
\end{split}
\end{equation*}
We could observe that the right-hand-side term is nondecreasing in $t$, hence
\begin{equation*}
\begin{split}
\sup _{0\leq u \leq t}\mathbb{E}|x_\Delta (u)|^2\leq C\left ( 1+\int _0^t\sup _{0\leq u \leq s}\mathbb{E}|x_\Delta (u)|^2ds\right ).
\end{split}
\end{equation*}
An application of Gronwall's inequality yields that
\begin{equation*}
\begin{split}
\sup _{0\leq u \leq T}\mathbb{E}|x_\Delta (u)|^2\leq C,
\end{split}
\end{equation*}
where $C$ is independent of $\Delta$.  We complete the proof.
\end{proof}

Since the boundedness of the numerical solution, the estimates of stopping times in  Lemma \ref{L9} still hold.
Now, we are going to state  the convergence of the truncated EM method for SDDEwPJs in $\mathcal{L}^p$ for $0<p <2$.
\begin{thm}\label{BL8}
Let Assumptions \ref{2.1}, \ref{a4.1} and \ref{a4.2} hold. Then, for any $p \in \left ( 0,2 \right )$, we have
\begin{equation}\label{es19}
\lim _{\Delta \rightarrow 0} \mathbb{E} |x(T)-x_\Delta(T)|^p =0,
\end{equation}
and
\begin{equation}\label{es20}
\lim _{\Delta \rightarrow 0} \mathbb{E} |x(T)-\bar{x}_\Delta(T)|^p =0.
\end{equation}
\end{thm}
\begin{proof}
Let $e_\Delta(t)=x(t)-x_\Delta(t)$ for $t\geq 0$ and $\Delta\in(0,1]$. Define $\rho _{\Delta,L} = \tau _L \wedge \tau _{\Delta,L}$. We write $\rho _{\Delta,L}=\rho$ for simplicity. Obviously,
\begin{equation}\label{es11}
\begin{split}
\mathbb{E} |e_\Delta (T)|^p
= \mathbb{E} \left ( |e_\Delta (T)|^p \textbf{I}_{\{\rho >T\}}  \right )+\mathbb{E} \left ( |e_\Delta (T)|^p \textbf{I}_{\{\rho \leq T\}}  \right ).
\end{split}
\end{equation}
Let $\delta >0$ be arbitrary. By Young's inequality, we have
\begin{equation}\label{es12}
\begin{split}
u^p v=(\delta u ^2)^{\frac{p}{2}} \left ( \frac{v ^{2/(2-p)}}{\delta^{p/(2-p)}}\right ) ^{\frac{2-p}{2}}\leq \frac{p\delta}{2} u^2 + \frac{2-p}{2\delta^{p/(2-p)}} v^{2/(2-p)}
,~~~\forall u,v >0.
\end{split}
\end{equation}
Hence,
\begin{equation}\label{es13}
\begin{split}
\mathbb{E} \left ( |e_\Delta (T)|^p \textbf{I}_{\{\rho \leq T\}}  \right )\leq\frac{p\delta}{2}\mathbb{E} |e_\Delta (T)|^2 + \frac{2-p}{2\delta ^ {p/(2-p)}}\mathbb{P} \{ \rho \leq T \}.
\end{split}
\end{equation}
Applying Lemma \ref{BL3} and \ref{BL6} yields that
\begin{equation}\label{es14}
\begin{split}
\mathbb{E} |e_\Delta (T)|^2 \leq C.
\end{split}
\end{equation}
By Lemma \ref{L9}, we have
\begin{equation}\label{es15}
\begin{split}
\mathbb{P}(\rho \leq T)\leq \mathbb{P}(\tau _{L} \leq T) +\mathbb{P}(\tau _{\Delta,L} \leq T) \leq \frac{C}{L^2}.
\end{split}
\end{equation}
Inserting \eqref{es14} and \eqref{es15} into \eqref{es13} yields that
\begin{equation}\label{es16}
\begin{split}
\mathbb{E} \left ( |e_\Delta (T)|^p \textbf{I}_{\{\rho \leq T\}}  \right )\leq\frac{C p\delta}{2} + \frac{C(2-p)}{2 L^2 \delta ^ {p/(2-p)}}.
\end{split}
\end{equation}
Let $\varepsilon$ be arbitrary. We choose $\delta$ sufficiently small such that
\begin{equation*}
\frac{C p\delta}{2} \leq \frac{\varepsilon}{3},
\end{equation*}
and choose $L$ sufficiently large such that
\begin{equation*}
\frac{C(2-p)}{2 L^2 \delta ^ {p/(2-p)}} \leq \frac{\varepsilon}{3}.
\end{equation*}
Thus,
\begin{equation}\label{es17}
\begin{split}
\mathbb{E} \left ( |e_\Delta (T)|^p \textbf{I}_{\{\rho \leq T\}}\right )\leq\frac{2\varepsilon}{3}.
\end{split}
\end{equation}
Moreover, we could use the similar technique in the proof of Theorem 3.5 in \cite{9} to prove that
\begin{equation}\label{es18}
\begin{split}
\mathbb{E} \left ( |e_\Delta (T)|^p \textbf{I}_{\{\rho > T\}}  \right )\leq \frac{\varepsilon}{3}.
\end{split}
\end{equation}
Combing \eqref{es11}, \eqref{es17} and \eqref{es18} together, we have
\begin{equation}
\begin{split}
\mathbb{E} |e_\Delta (T)|^p \leq \varepsilon .
\end{split}
\end{equation}
Hence, we get the desired result \eqref{es19}. Then combing \eqref{es19} and Lemma \ref{BL5} yields \eqref{es20}. We complete the proof.
\end{proof}

Next, in order to estimate the rate of the convergence at time $T$, we have to impose an extra condition. \\
\begin{ass}\label{a4.9}
There exists a positive constant $K_7 $ such that
\begin{equation}
\begin{split}
&2(x-\bar{x})^T (f(x,y)-f(\bar{x},\bar{y}))+|g(x,y)-g(\bar{x},\bar{y})|^2\\
&+2\lambda(x-\bar{x})^T (h(x,y)-h(\bar{x},\bar{y}))+\lambda|h(x,y)-h(\bar{x},\bar{y})|^2\\
&\leq K_7(|x-\bar{x}|^2+|y-\bar{y}|^2)-U(x,\bar{x})+U(y,\bar{y})
\end{split}
\end{equation}
for any $x,y,\bar{x},\bar{y} \in \mathbb{R} ^n $. Here, $U(\cdot ,\cdot)$ is defined as before.
\end{ass}
\begin{lem}\label{BL10}
Let Assumptions \ref{2.1}, \ref{a4.1}, \ref{a4.2} and \ref{a4.9} hold. Let $\Delta \in (0,1)$ be sufficiently small such that $\varphi ^{-1} (\alpha(\Delta)) \geq L\vee |\sup_{-\tau\le s \le 0}|\xi(s)|$. Then we have
\begin{equation}
\begin{split}
\mathbb{E}|x(T\wedge \rho_{\Delta,L})-x_\Delta(T\wedge \rho_{\Delta,L})|^2 \leq C\left ( (\alpha(\Delta))^2 \Delta ^{\frac{1}{2}} +\Delta ^{2\gamma}\right ),
\end{split}
\end{equation}
where $\rho_{\Delta,L}:= \tau _L \wedge \tau _{\Delta,L}$ and $\tau _L$, $\tau _{\Delta,L}$ is defined in Lemma \ref{L9}.
\end{lem}
\begin{proof}
Let $e_\Delta(t)=x(t)-x_\Delta(t)$ for $t\geq 0$ and $\Delta\in(0,1]$. We write $\rho _{\Delta,L}=\rho$ for simplicity. For $0\leq s \leq t \wedge \rho$, we observe that
\begin{equation}
\begin{split}
|x(s)|\vee|x(s-\tau)|\vee|\bar{x}_\Delta(s)|\vee|\bar{x}_\Delta(s-\tau)| \leq L \leq \varphi ^{-1} (\alpha(\Delta)).
\end{split}
\end{equation}
Recalling the definition of $f_\Delta$, $g_\Delta$ and $h_\Delta$, we obtain for $0\leq s \leq t \wedge \rho$ that
\begin{equation*}
\begin{split}
f_\Delta (\bar{x}_\Delta (s),\bar{x}_\Delta (s-\tau))=f(\bar{x}_\Delta (s),\bar{x}_\Delta (s-\tau)),
\end{split}
\end{equation*}
\begin{equation*}
\begin{split}
g_\Delta (\bar{x}_\Delta (s),\bar{x}_\Delta (s-\tau))=g(\bar{x}_\Delta (s),\bar{x}_\Delta (s-\tau)),
\end{split}
\end{equation*}
\begin{equation*}
\begin{split}
h_\Delta (\bar{x}_\Delta (s),\bar{x}_\Delta (s-\tau))=h(\bar{x}_\Delta (s),\bar{x}_\Delta (s-\tau)),
\end{split}
\end{equation*}
and
\begin{equation}
\begin{split}
&|f(x(s),x(s-\tau))|\vee|f(\bar{x}_\Delta (s),\bar{x}_\Delta (s-\tau))|\vee|h(x(s),x(s-\tau))|\\
&\vee|h(\bar{x}_\Delta (s),\bar{x}_\Delta (s-\tau))|\leq \alpha(\Delta),
\end{split}
\end{equation}

By It\^{o}'s formula and Assumption \ref{a4.9}, for any $0\leq t \leq T$ we have
\begin{equation}\label{ks0}
\begin{split}
&~~~~\mathbb{E} |e_\Delta(t \wedge \rho)|^2 \\
&\leq \mathbb{E} \int _0^{t \wedge \rho} \Big (2 (x(s)-\bar{x}_\Delta (s))^T (f(x(s),x(s-\tau))-f_\Delta (\bar{x}_\Delta (s),\bar{x}_\Delta (s-\tau))) \\
&~~~ + |g(x(s),x(s-\tau))- g_\Delta (\bar{x}_\Delta (s),\bar{x}_\Delta (s-\tau))|^2\\
&~~~+2\lambda (x(s)-\bar{x}_\Delta (s))^T (h(x(s^-),x((s-\tau)^-))-h_\Delta(\bar{x}_\Delta(s^-),\bar{x}_\Delta((s-\tau)^-)))\\
&~~~+ \lambda |h(x(s^-),x((s-\tau)^-))-h_\Delta(\bar{x}_\Delta(s^-),\bar{x}_\Delta((s-\tau)^-))|^2
 \Big )ds\\
&~~~ + \mathbb{E}\int _0^{t \wedge \rho} 2 (\bar{x}_\Delta (s)-x_\Delta (s))^T (f(x(s),x(s-\tau))-f_\Delta (\bar{x}_\Delta (s),\bar{x}_\Delta (s-\tau)))ds\\
&~~~ + \mathbb{E}\int _0^{t \wedge \rho} 2 \lambda(\bar{x}_\Delta (s)-x_\Delta (s))^T (h(x(s^-),x((s-\tau)^-))-h_\Delta(\bar{x}_\Delta(s^-),\bar{x}_\Delta((s-\tau)^-)))ds\\
&\leq \mathbb{E}\int _0^{t \wedge \rho} K_7(|x(s)-\bar{x}_\Delta (s)|^2+|x(s-\tau)-\bar{x}_\Delta (s-\tau)|^2)ds\\
&~~~+\mathbb{E}\int _0^{t \wedge \rho} (-U(x(s),\bar{x}_\Delta (s))+U(x(s-\tau),\bar{x}_\Delta (s-\tau)))ds\\
&~~~+\mathbb{E}\int _0^{t \wedge \rho} 2 |\bar{x}_\Delta (s)-x_\Delta (s)||f(x(s),x(s-\tau))-f_\Delta (\bar{x}_\Delta (s),\bar{x}_\Delta (s-\tau))|ds\\
&~~~+\mathbb{E}\int _0^{t \wedge \rho} 2 \lambda|\bar{x}_\Delta (s)-x_\Delta (s)||h(x(s^-),x((s-\tau)^-))-h_\Delta(\bar{x}_\Delta(s^-),\bar{x}_\Delta((s-\tau)^-))|ds\\
&=: J_1 +J_2 +J_3 +J_4.
\end{split}
\end{equation}
By Assumption \ref{2.1} and Lemma \ref{BL5}, similar to the proof of Theorem \ref{thm1},  we derive that
\begin{equation}\label{ks1}
\begin{split}
J_1 &\leq 4 K_7\int _0^{t}\mathbb{E}|e_\Delta(s\wedge \rho)|^2ds+C((\alpha(\Delta))^2 \Delta +\Delta^{2\gamma}),\\
J_2 &\leq \int _{-\tau}^{0}U(\xi(s),\xi(\kappa(s)))ds \leq \int _{-\tau}^{0}\kappa_b|\xi(s)-\xi(\kappa(s))|^2ds
\leq \tau \kappa_b \bar{K} ^2\Delta^{2\gamma},\\
J_3&\leq 4\alpha(\Delta)\mathbb{E}\int _0^{t \wedge \rho} |x_\Delta (s)-\bar{x}_\Delta (s)|ds\leq C(\alpha(\Delta))^2 \Delta ^{\frac{1}{2}},\\
J_4&\leq C(\alpha(\Delta))^2 \Delta ^{\frac{1}{2}}.
\end{split}
\end{equation}
These imply that
\begin{equation*}
\begin{split}
\mathbb{E} |e_\Delta(t \wedge \rho)|^2
\leq C\left (\int _0^{t}\mathbb{E}|e_\Delta(s\wedge \rho)|^2ds+(\alpha(\Delta))^2 \Delta ^{\frac{1}{2}} +\Delta^{2\gamma} \right ).
\end{split}
\end{equation*}
The required assertion follows by the Gronwall inequality.
\end{proof}
\begin{thm}\label{BL11}
Let Assumptions \ref{2.1}, \ref{a4.1}, \ref{a4.2} and \ref{a4.9} hold. Let $p \in (0,2)$, for any sufficiently small $\Delta\in(0,1)$, assume that there exists a positive constant $c_2$ such that
\begin{equation}\label{ks5}
\begin{split}
\alpha(\Delta) \geq \varphi \left ( c_2([(\alpha(\Delta))^p \Delta ^{p/4}]\vee \Delta ^{p\gamma})^{-1/(2-p)}\right ).
\end{split}
\end{equation}
 Then,  for any $T>0$, we have that
\begin{equation}\label{ks6}
\mathbb{E}|x(T)-x_{\Delta}(T)|^p \leq C\left ( [(\alpha(\Delta))^p \Delta ^{p/4}]\vee[\Delta ^{p\gamma}]
\right ),
\end{equation}
and
\begin{equation}\label{ks7}
\mathbb{E}|x(T)-\bar{x}_{\Delta}(T)|^p \leq C\left ( [(\alpha(\Delta))^p \Delta ^{p/4}]\vee[\Delta ^{p\gamma}]
\right ).
\end{equation}

\end{thm}
\begin{proof}
We use the notation of $e_\Delta (t)$ and $\rho _{\Delta,L}$ as before. We write $\rho _{\Delta,L}=\rho$ for simplicity. By \eqref{es11} and \eqref{es16}, one can see  that
\begin{equation}
\begin{split}
\mathbb{E}  |e_\Delta (T)|^p
&=\mathbb{E} \left ( |e_\Delta (T)|^p \textbf{I}_{\{\rho > T\}}  \right )+\mathbb{E} \left ( |e_\Delta (T)|^p \textbf{I}_{\{\rho \leq T\}}  \right )\\
&\leq \mathbb{E} |e_\Delta (T\wedge \rho)|^p  +\frac{C p\delta}{2} + \frac{C(2-p)}{2 L^2 \delta ^ {p/(2-p)}},
\end{split}
\end{equation}
for any $\Delta\in(0,1)$, $L>\|\xi\|$ and $\delta >0$.  Choosing $\delta = [(\alpha(\Delta))^p \Delta ^{p/4}]\vee[\Delta ^{p\gamma}]$ and $L=c_2 ([(\alpha(\Delta))^p \Delta ^{p/4}]\vee \Delta ^{p\gamma})^{-1/(2-p)}$, we obtain
\begin{equation}
\begin{split}
\mathbb{E}  |e_\Delta (T)|^p
&\leq \mathbb{E} |e_\Delta (T\wedge \rho)|^p  +C \left ([(\alpha(\Delta))^p \Delta ^{p/4}]\vee[\Delta ^{p\gamma}] \right ).
\end{split}
\end{equation}
By the condition \eqref{ks5}, we derive that
\begin{equation}
\begin{split}
\varphi ^{-1}(\alpha(\Delta)) \geq c_2([(\alpha(\Delta))^p \Delta ^{p/4}]\vee \Delta ^{p\gamma})^{-1/(2-p)}=L.
\end{split}
\end{equation}
Using  Lemma \ref{BL10}, one has that
\begin{equation}
\begin{split}
&\mathbb{E}|x(T)-x_{\Delta}(T)|^p \leq \left (\mathbb{E}|x(T)-x_{\Delta}(T)|^2\right )^{\frac{p}{2}}\\
&\leq C\left ( [(\alpha(\Delta))^2 \Delta ^{\frac{1}{2}}]\vee[\Delta ^{2\gamma}]
\right )^{\frac{p}{2}}
\leq C\left ( [(\alpha(\Delta))^p \Delta ^{p/4}]\vee[\Delta ^{p\gamma}]
\right ).
\end{split}
\end{equation}
Combing Lemma \ref{BL5} and \eqref{ks6} together, we can derive \eqref{ks7}. We complete the proof.
\end{proof}
\begin{rem}
If we impose an additional condition: Assume that there exist constants $K_8 >0$ and $\bar{\beta} \in [0,1)$ such that
\begin{equation}\label{ks8}
\begin{split}
&|f(x,y)-f(\bar{x},\bar{y})|\vee |h(x,y)-h(\bar{x},\bar{y})| \\
&\leq K_8(1+|x|^{\bar{\beta}} +|y|^{\bar{\beta}} +|\bar{x}|^{\bar{\beta}} +|\bar{y}|^{\bar{\beta}})(|x-\bar{x}|+|y-\bar{y}|)\\
\end{split}
\end{equation}
for any $x,y,\bar{x},\bar{y} \in \mathbb{R} ^n $, we could obtain better convergence rate. But our main result can cover more equations without this condition \eqref{ks8}.
\end{rem}

\section{Example }

In this section, we give an example to illustrate our theories. Consider the super-linear scalar SDDEwPJs
\begin{equation}\label{ms1}
\begin{split}
&dx(t) = (-5x^3 (t)+\frac{1}{8}|x(t-\tau)|^{\frac{5}{4}}+2x(t))dt +(\frac{1}{2}|x(t)|^{\frac{3}{2}} +x(t-\tau))dB(t)\\
&\qquad +(x(t^-)+x((t-\tau)^-))dN(t),
\end{split}
\end{equation}
with the initial value $\xi =\{x(\theta):-\tau \leq \theta \leq 0\}$ which satisfies Assumption \ref{2.1}. Here $B(t)$ is a scalar Brownian motion and $N(t)$ is a scalar Poisson process with intensity $\lambda =0.2$.

Now we are verifying the Assumption \ref{a1} - \ref{a3}.

It is easy to see that
\begin{equation}\label{ms2}
\begin{split}
&|f(x,y)-f(\bar{x},\bar{y})|\vee|g(x,y)-g(\bar{x},\bar{y})|\\
& = |(-5x^3+\frac{1}{8}|y|^{\frac{5}{4}}+2x)-(-5\bar{x}^3+\frac{1}{8}|\bar{y}|^{\frac{5}{4}}+2\bar{x})| \vee|(\frac{1}{2}|x|^{\frac{3}{2}} +y)-(\frac{1}{2}|\bar{x}|^{\frac{3}{2}} +\bar{y})|\\
&\leq 10 (1+|x|^2 +|y|^2 +|\bar{x}|^2 +|\bar{y}|^2)(|x-\bar{x}|+|y-\bar{y}|),
\end{split}
\end{equation}
and
\begin{equation}\label{ms3}
\begin{split}
&|h(x,y)-h(\bar{x},\bar{y})|=|(x+y)-(\bar{x}+\bar{y})|\leq 10 (|x-\bar{x}|+|y-\bar{y}|).
\end{split}
\end{equation}
Hence, Assumption \ref{a1} is satisfied with $\beta =2$. Moreover,  we can see that
\begin{equation}\label{ms4}
\begin{split}
&(x-\bar{x})^T (f(x,y)-f(\bar{x},\bar{y}))\\
&=5(x-\bar{x})^2(-(x^2 +x\bar{x}+\bar{x}^2)) +2(x-\bar{x})^2 +\frac{1}{8}(x-\bar{x})(|y|^{\frac{5}{4}}-|\bar{y}|^{\frac{5}{4}})\\
&\leq 5(x-\bar{x})^2(-\frac{1}{2}(x^2 +\bar{x}^2))+3(x-\bar{x})^2+\frac{25}{256}|y-\bar{y}|^2(|y|^{\frac{1}{4}}+|\bar{y}|^{\frac{1}{4}})^2\\
&\leq -\frac{5}{2}|x-\bar{x}|^2(|x|^2+|\bar{x}|^2) +3|x-\bar{x}|^2 +\frac{25}{64}|y-\bar{y}|^2+\frac{25}{128}|y-\bar{y}|^2(|y|^2+|\bar{y}|^2).
\end{split}
\end{equation}
Let $\bar{\eta}=3$. In the same way, we can derive
\begin{equation}\label{ms5}
\begin{split}
&\frac{\bar{\eta} -1}{2} |g(x,y)-g(\bar{x},\bar{y})|^2
=|(\frac{1}{2}|x|^{\frac{3}{2}} +y)-(\frac{1}{2}|\bar{x}|^{\frac{3}{2}} +\bar{y})|^2\\
&\leq\frac{1}{2}||x|^{\frac{3}{2}} -|\bar{x}|^{\frac{3}{2}}|^2+2|y-\bar{y}|^2
\leq\frac{9}{8}|x-\bar{x}|^2(|x|^{\frac{1}{2}}+|\bar{x}|^{\frac{1}{2}})^2+2|y-\bar{y}|^2\\
&\leq\frac{9}{2}|x-\bar{x}|^2+\frac{9}{4}|x-\bar{x}|^2(|x|^2+|\bar{x}|^2)+2|y-\bar{y}|^2.
\end{split}
\end{equation}
Combing \eqref{ms4} and \eqref{ms5} gives that
\begin{equation}\label{ms6}
\begin{split}
&(x-\bar{x})^T (f(x,y)-f(\bar{x},\bar{y}))+\frac{\bar{\eta} -1}{2} |g(x,y)-g(\bar{x},\bar{y})|^2\\
&\leq 8(|x-\bar{x}|^2+|y-\bar{y}|^2)-\frac{1}{4}|x-\bar{x}|^2(|x|^2+|\bar{x}|^2)
+\frac{1}{4}|y-\bar{y}|^2(|y|^2+|\bar{y}|^2).
\end{split}
\end{equation}
Therefore, Assumption \ref{a2} is satisfied with $U(x,\bar{x})=\frac{1}{4}|x-\bar{x}|^2(|x|^2+|\bar{x}|^2)$. Noting that
\begin{equation}\label{ms7}
\begin{split}
&x^T f(x,y) +\frac{\bar{p} -1}{2} |g(x,y)|^2
=x(-5x^3+\frac{1}{8}|y|^{\frac{5}{4}}+2x) +\frac{\bar{p} -1}{2}|\frac{1}{2}|x|^{\frac{3}{2}} +y|^2\\
&\leq C(1+|x|^2 +|y|^2).
\end{split}
\end{equation}
 Thus, Assumption \ref{a3} is satisfied as well.

Additionally, it is easy to see that
\begin{equation}\label{ms8}
\sup_{|x|\vee|y| \le r} \left ( |f(x,y)| \vee |g(x,y)| \vee |h(x,y)|\right ) \le 5r^3,~~~\forall r\geq 1.
\end{equation}
hence, we can choose $\varphi (r)=5r^3$. This means $\varphi ^{-1} (r)=(\frac{r}{5})^{\frac{1}{3}}$. In order for  $q\geq \frac{1+\beta}{\varepsilon}$ to hold, we set $\bar{p}=26$ such that Assumption \ref{a3} be satisfied. Then $q\in ((1+\beta)\bar{\eta},\bar{p})=q\in (9,26)$. We choose $q=25$. Moreover, let $\varepsilon =\frac{1}{8}$, then $q\geq \frac{1+\beta}{\varepsilon}$ is satisfied. So  $\alpha(\Delta) = K_0 \Delta ^{-\frac{1}{8}}$. By Theorem \ref{thm1}, we have
\begin{equation}\label{ms9}
\mathbb{E}|x(T)-x_{\Delta}(T)|^2 \leq C\Delta^{2\gamma},~~~\forall \gamma \in (0,\frac{3}{8}),
\end{equation}
and
\begin{equation}\label{ms10}
\mathbb{E}|x(T)-x_{\Delta}(T)|^2 \leq C\Delta^{\frac{3}{4}},~~~\forall \gamma \in [\frac{3}{8},1],
\end{equation}
which means that the $\mathcal{L}^2$-convergence rate of the truncated EM method for SDDEwPJs \eqref{ms1} is $(2\gamma)\wedge\frac{3}{4}$.

On the other hand, let us verify the Assumption \ref{a4.1}, \ref{a4.2} and \ref{a4.9}. By \eqref{ms2} and \eqref{ms3}, we find that Assumption \ref{a4.1} is satisfied. In addition, we have
\begin{equation}\label{ms11}
\begin{split}
&2x^T f(x,y) +|g(x,y)|^2 +\lambda(2x^T h(x,y) +|h(x,y)|^2)\\
&=2x(-5x^3+\frac{1}{8}|y|^{\frac{5}{4}}+2x) +|\frac{1}{2}|x|^{\frac{3}{2}} +y|^2+\lambda(2x(x+y) +|x+y|^2)\\
&\leq C(1+|x|^2 +|y|^2).
\end{split}
\end{equation}
Thus, Assumption \ref{a4.2} is satisfied. By \eqref{ms4} and \eqref{ms5}, we obtain
\begin{equation}\label{ms12}
\begin{split}
&2(x-\bar{x})^T (f(x,y)-f(\bar{x},\bar{y}))+|g(x,y)-g(\bar{x},\bar{y})|^2\\
&+2\lambda(x-\bar{x})^T (h(x,y)-h(\bar{x},\bar{y}))+\lambda|h(x,y)-h(\bar{x},\bar{y})|^2\\
&\leq 11(|x-\bar{x}|^2+|y-\bar{y}|^2)-\frac{11}{4}|x-\bar{x}|^2(|x|^2+|\bar{x}|^2)+\frac{11}{4}|y-\bar{y}|^2(|y|^2+|\bar{y}|^2)\\
&~~~+5\lambda|x-\bar{x}|^2+3\lambda|y-\bar{y}|^2\\
&\leq 12 (|x-\bar{x}|^2+|y-\bar{y}|^2)-\frac{11}{4}|x-\bar{x}|^2(|x|^2+|\bar{x}|^2) +\frac{11}{4}|y-\bar{y}|^2(|y|^2+|\bar{y}|^2).
\end{split}
\end{equation}
Hence, Assumption \ref{a4.9} is satisfied with $U(x,\bar{x})=\frac{11}{4}|x-\bar{x}|^2(|x|^2+|\bar{x}|^2)$.

Choose $\varphi(r)=5r^3$, $c_2=(\frac{1}{5})^{\frac{1}{3}}$. Let $0<p < \frac{2}{12\gamma +1}$ and
define
\begin{equation}\label{ms14}
\begin{split}
\alpha(\Delta) =\Delta^{-\varepsilon},~~~\forall \varepsilon \in \left[ \left( \frac{3p}{8(1+p)} \vee \frac{3p\gamma}{2-p} \right),\frac{1}{4}\right].
\end{split}
\end{equation}
Then condition \eqref{ks5} is satisfied, that is
\begin{equation*}
\begin{split}
\alpha(\Delta) \geq \varphi \left ( c_2([(\alpha(\Delta))^p \Delta ^{p/4}]\vee \Delta ^{p\gamma})^{-1/(2-p)}\right ).
\end{split}
\end{equation*}
By Theorem \ref{BL11}, we have
\begin{equation}\label{ms15}
\mathbb{E}|x(T)-x_{\Delta}(T)|^p \leq C\Delta^{p\left( (\frac{1}{4}-\varepsilon)\wedge \gamma\right)},
\end{equation}
which means that the $\mathcal{L}^p$-convergence($p \in (0,2)$) rate of the truncated EM method for SDDEwPJs \eqref{ms1} is $p\left( (\frac{1}{4}-\varepsilon)\wedge \gamma\right)$.

\section*{Acknowledgements}

This work is supported by the National Natural Science Foundation of China (Grant No. 61876192), NSF of Jiangxi(Grant Nos. 20192ACBL21007, 2018ACB21001), the Fundamental Research Funds for the Central Universities(CZT20020) and Academic Team in Universities(KTZ20051).

\end{document}